\documentclass[reqno]{amsart}

\usepackage{amssymb,latexsym}
\usepackage{amsfonts}
\usepackage{dsfont}
\usepackage[pdftex]{hyperref}
\usepackage{color}
\usepackage{graphicx}
\usepackage{parskip}
\usepackage{mathrsfs}

\theoremstyle{plain}

\newtheorem{theorem}{Theorem}[section]
\newtheorem{lemma}{Lemma}[section]

\newtheorem{proposition}{Proposition}[section]
\newtheorem{defin}{Definition}[section]

\newtheorem{corollary}{Corollary}[section]
\numberwithin{equation}{section}

\newcommand{\p}{\partial}
\newcommand{\m}{\mathbb}

\newcommand{\D}{\Delta}

\newcommand{\rr}{\mathbb{R}}
\newcommand{\nn}{\mathbb{N}}

\newcommand{\vp}{\varphi}

\newcommand{\ms}{\mathcal{S}}
\newcommand{\mf}{\mathscr{F}}

\begin{document}

\title[The FORQ equation, J. Holmes: holmesj@wfu.edu]{Continuity of the data-to-solution map for the FORQ equation in Besov Spaces}
\author [The FORQ equation, J. Holmes: holmesj@wfu.edu] 
{John Holmes$^1$,  Feride T\i\u glay$^2$, Ryan Thompson $^3$
\\
$^1${\em \tiny {Department of Mathematics and Statistics, Wake Forest University, Winston Salem, NC 27109, USA}}
\\
$^2${\em \tiny {Department of Mathematics, The Ohio State University, Columbus, OH 43210, USA}}
\\
$^3${\em \tiny {Department of Mathematics, The University of North Georgia, Dahlonega, GA 30597, USA}}
}
 \date{%
    \today
}

\begin{abstract} 
For Besov spaces $B^s_{p,r}(\rr)$ with $s>\max\{ 2 + \frac1p , \frac52\} $,  $p \in (1,\infty]$ and $r \in [1 , \infty)$, it is proved that the data-to-solution map for the FORQ equation is not uniformly continuous from $B^s_{p,r}(\rr)$ to $C([0,T]; B^s_{p,r}(\rr))$. The proof of non-uniform dependence is based on approximate solutions and the Littlewood-Paley decomposition.  
\end{abstract}


\maketitle

%
%
%
%
\section{Introduction} 
We consider the following initial value problem for the integrable partial differential equation, known as the FORQ equation: 
\begin{align}\label{forq}
&u_t  + u^2 u_x - u_x^3+  (1-\p_x^2)^{-1} \partial_x ( \frac23 u ^3  +uu_x^2 ) +    (1-\p_x^2)^{-1} (\frac13 u_x^3 )= 0\\
& u(x,0) = u_0(x),\label{forq-data}
\end{align}
where
 $x\in \mathbb R$ and $u_0 \in B^s_{p,r} (\mathbb R)=  B^s_{p,r} $ where $(1-\p_x^2)^{-1} f = \mathcal F^{-1} ( \frac{1}{1+\xi^2} \widehat f (\xi))$. 
The FORQ equation was derived in  Fokas \cite{F}, Fuchssteiner \cite{Fu}, Olver and Rosenau \cite{OR} by applying the method of tri-Hamiltonian duality to the bi-Hamiltonian representation of the modified Korteweg-de Vries equation. It was also derived by Qiao \cite{Q} who showed an entire integrable hierarchy known as the MCH hierarchy and the existence of a Lax pair.  Qiao also showed that FORQ admits solitary solutions including {\em peakons}, solitary traveling wave  solutions of the form
$$
u(x,t) = \sqrt{\frac{c}{2}} e^{-|x-ct| } .
$$

    Peakons  were discovered  by Fornberg and Whitham \cite{FW} and then by Camassa and Holm \cite{CH} as solutions to shallow water wave equations, and take the general form
    $$
    u(x,t) = ce^{-|x-q(t)|}.
    $$
     The Camassa-Holm (CH) equation 
    \begin{align}
    (1-\partial_x)^2 u_t = uu_{xxx} +2u_x u_{xx} - 3uu_x,
    \end{align} 
    is an integrable equation closely related to the FORQ equation, which also admits peakons; for more information about peakon solutions, we refer the reader to Holm
and Ivanov \cite{HI}.  Originally derived by Fokas and Fuchssteiner \cite{FF} in the context of hereditary symmetries, it was later derived by \cite{CH} from the Euler equations through asymptotic expansions. 
 One may note that the CH equation has quadratic rather than cubic nonlinearities, and this plays an important role in the analysis of these two equations. Local well-posedness results for the CH equation were discovered by Li and Olver \cite{LO}, Constantine and Escher \cite{CE}, Rodriguez-Blanco \cite{RB} and Danchin \cite{D2001} among others.

Well-posedness in the sense of Hadamard for the Cauchy problem \eqref{forq}-\eqref{forq-data} was  shown by Himonas and Matzavinos  \cite{HMA} in Sobolev spaces $H^s$, $s>5/2$ in both the periodic and non periodic cases. 
Well-posedness   in Besov spaces $B^s_{p,r}$ with $p \in [1,\infty]$, $r \in [1 , \infty)$ and $s > \max\{ 2 + \frac1p , \frac52\} $ and at the critical index corresponding to $s = 5/2$, $p = 2$, $r=1$, by Fu, Gui, Liu and Qu \cite{FGLQ} (additionally the case $r=\infty$ is considered, however, the continuity of the data-to-solution map is established in a weaker topology). These authors also derive a precise blow-up scenario and a lower bound on the maximal time of existence for solutions. The methodology used in the well-posedness argument in Besov spaces is based upon the ideas in \cite{D2001}.

Ill-posedness for the Cauchy problem \eqref{forq} in $H^s$ for any $s<3/2$ was demonstrated by studying the interaction of multi-peakon solutions in  Himonas and Holliman \cite{AH}. This was extended to a related family of equations in \cite{holmes2020}. Well-posedness in the gap, $\frac32<s<\frac52$, remains an open question for the FORQ initial value problem, and appears to be  challenging  due to  the $u_x^3$ term.

Continuity of the data-to-solution map is an important part of the well-posedness theory, and is very delicate for the CH equation and related problems. In fact, Himonas and Kenig \cite{HK}
and Himonas, Kenig and Misiolek \cite{HKM} prove that the data-to-solution map for the CH equation in Sobolev spaces is not uniformly continuous in the non-periodic and periodic cases respectively. They use the method of approximate solutions, conservation of the $H^1$ norm, and commutator estimates. More recently, a similar result has been shown in Besov spaces by Li, Yu and Zhu \cite{LYZ}. This result does not extend cleanly to the FORQ equation. Rather, using some ideas from \cite{LYZ} as well as some ideas from  \cite{HMA}  we are able to deal with the cubic nonlinearities in the FORQ equation  and prove nonuniform dependence of the data-to-solution map in Besov spaces. In particular, our main
 result can be stated as follows. 
\begin{theorem}\label{thm1} 
Assume $s>\max\{ 2 + \frac1p , \frac52\} $, with $p \in (1,\infty]$ and $r \in [1 , \infty)$
. Then the data-to-solution map for the initial value problem \eqref{forq}-\eqref{forq-data} is not uniformly continuous  from $B^s_{p,r}(\rr)$ to $C([0,T]; B^s_{p,r}(\rr))$. 
\end{theorem} 
The case $p=1$ is not covered by our theorem. Similar to the difficulties in proving well-posedness, our estimates fail in this case due to the $u_x^3$ term in the FORQ equation. 

Our paper is organized as follows. In the next section, Section 2,  we present some preliminary results and introduce notation used in Besov space. Section 3 presents the proof of Theorem  \ref{thm1} and introduces some important estimates readers may find useful for  related problems. In particular we introduce two sequences of functions which we will use to construct our approximate solutions. We next provide some necessary estimates concerning these functions, and  then outline our proof before providing the details. 
%
%
%
%
\section{Preliminary result and notation}
In this section, we will recall some conclusions on the properties of Littlewood-Paley decomposition, and Besov spaces; these results may be found in \cite{danchin-book}. We begin with the Littlewood-Paley decomposition. 

\begin{lemma}
(Littlewood-Paley decomposition).  There exists a couple of smooth radial functions $(\chi,\vp)$ valued in $[0,1]$ such that $\chi$ is supported in the ball $B=\{\xi \in \m{R}^n,|\xi|\leq \frac43\}$ and $\vp$ is supported in the ring $C=\{\xi \in \m{R}^n,\frac34\leq|\xi|\leq \frac83\}$.  Moreover,
\[
\forall \xi \in \m{R}^n, \ \ \ \chi(\xi)+\sum_{q \in \m{N}}\vp(2^{-q}\xi)=1
\]
and
\[
supp \ \vp(2^{-q}\cdot) \ \cap \ supp \ \vp(2^{-q'}\cdot) = \emptyset, \ \ if \ \ |q-q'|\geq 2,
\]
\[
supp \ \chi(\cdot) \ \cap \ supp \ \vp(2^{-q}\cdot) = \emptyset, \ \ if \ \ |q|\geq 1.
\]
Then for $u \in \mathcal{S}'(\m{R})$ the nonhomogeneous dyadic blocks are defined as follows:
\begin{align*}
&\Delta_qu=0, \ \ if \ \ q \leq -2, \\
&\Delta_{-1}u=\chi(D)u = \mathscr{F}_x^{-1}\chi\mathscr{F}u, \\
&\Delta_qu=\vp(2^{-q}D)=\mathscr{F}_x^{-1}\vp(2^{-q}\xi)\mathscr{F}u, \ \ if \ \ q\geq 0.
\end{align*}
Thus $u=\sum_{q \in \m{Z}}\Delta_qu$ in $\mathcal{S}'(\m{R})$.
\end{lemma}
\textbf{Remark.}  The low frequency cut-off $S_q$ is defined by
\[
S_qu=\sum_{p=-1}^{q-1}\Delta u =  \mf_x^{-1}\chi(2^{-q}\xi)\mf_xu, \ \ \forall q \in \nn.
\]
We can see that
\[
\D_p\D_qu\equiv 0, \ \ if \ \ |p-q|\geq 2,
\]
\[
\ \ \ \ \ \  \ \ \ \ \D_q(S_{p-1}u\D_pv)\equiv 0, \ \ if \ \ |p-q| \geq 5, \ \ \forall u,v \in \ms'(\rr)
\]
as well as
$
\|\D_qu\|_{L^p} \leq \|u\|_{L^p}, $ and $\|S_qu\|_{L^p} \leq C\|u\|_{L^p},$ $ \forall p \in [1,\infty]$,
with the aid of Young's Inequality, where $C$ is a positive constant independent of $q$.

Using the Littlewood-Paley decomposition, we may now define the Besov space. 
\begin{defin}
(Besov Spaces)
Let $s \in \rr$, $p,r \in [1 , \infty]$ and $u \in S'(\rr)$. 
Then we define the Besov space of functions as
$$
B^s_{p,r} = B^s_{p,r}(\rr) = \{ u \in S'(\rr) : \|u\|_{B^s_{p,r} }<\infty\} , 
$$
where
$$
\| u\|_{B^s_{p,r} }\dot = 
\begin{cases}
\left( \sum_{q \ge -1} (2^{sq} \| \Delta _q u \|_{L^p})^r \right)^{1/r} &\text{ if } 1 \le r< \infty
\\
\sup_{q\ge -1} 2^{sq}\|\Delta _q u \|_{L^p} &\text{ if } r= \infty.
\end{cases}
$$
In particular, $B_{p,r}^\infty=\bigcap_{s \in \rr}B_{p,r}^s$.
\end{defin}
There are several important, albeit standard  results which ensure Besov spaces are amenable to studying partial differential equations. 
\begin{lemma}\label{properties} 
Let $s \in \rr$, $1\leq p,r,p_j,r_j \leq \infty$, $j=1,2$, then
\begin{enumerate}
\item  Topological properties:  $B_{p,r}^s$ is a Banach space which is continuously emedded in $\ms'(\rr)$.

\item  Density:  $C_c^\infty$ is dense in $B_{p,r}^s$ $\iff$ $p,r \in [1,\infty)$. 

\item  Embedding:  $B_{p_1,r_1}^s \hookrightarrow B_{p_2,r_2}^{s-(\frac1p_1-\frac1p_2)}$, if $p_1 \leq p_2$ and $r_1 \leq r_2$.
\[
B_{p,r_2}^{s_2} \hookrightarrow B_{p,r_1}^{s_1}, \ \ \ \text{locally compact if} \ \ s_1<s_2.
\]

\item  Algebraic properties:  $\forall s >0$, $B_{p,r}^s\cap L^\infty$ is a Banach algebra.  $B_{p,r}^s$ is a Banach algebra $\iff$ $B_{p,r}^s \hookrightarrow L^\infty \iff s>\frac1p$ or $(s \geq \frac1p \  \text{and} \ r=1)$.  In particular, $B_{p,1}^{1/p}$ is continuously embedded in $B_{p,\infty}^{1/p} \cap L^\infty$ and $B_{p,\infty}^{1/p} \cap L^\infty$ is a Banach algebra.

\item  1-D Moser-type estimates:  
 For $s>0$,
\[
\|fg\|_{B_{p,r}^s} \leq C(\|f\|_{B_{p,r}^s}\|g\|_{L^\infty}+\|f\|_{L^\infty}\|g\|_{B_{p,r}^s}).
\]

\item 
For $s > \max \{ 1+\frac1p, \frac32 \}$, 
$$
\| fg\|_{B^{s-2}_{p,r} } \le C \| f\| _{B^{s-2}_{p,r} } \|g\| _{B^{s-1}_{p,r} }  .
$$

\item  Interpolation:  
\[
\|f\|_{B_{p,r}^{\theta s_1+(1-\theta)s_2}} \leq \|f\|_{B_{p,r}^{s_1}}^\theta\|g\|_{B_{p,r}^{s_2}}^{1-\theta}, \ \ \ \forall f \in B_{p,r}^{s_1} \cap B_{p,r}^{s_2}, \ \ \forall \theta \in [0,1].
\]

\end{enumerate}
\end{lemma}
Finally, we shall use the following transport equation estimate extensively.

\begin{proposition} \label{transport_prop} 
Let $p,r \in [1,\infty]$ and $\sigma >  1+ \frac1p$ or $\sigma = 1+\frac1p$ when 
$r=1$. Assume $v_0 \in B^\sigma _{p,r}$, $g \in L^1 ( [ 0,T]; B^\sigma _{p,r} )$ 
and $\partial _x f \in L^1 
( [0, T] ; B^{\sigma -1} _{p,r} $. Let $v \in L^\infty ([0, T] B^\sigma _{p,r})\cap C ([0,T]; \mathcal S ' )$ solve the following transport equation, 
\begin{align}
\begin{cases}
 v_t + f v_x = g\\
 v(x,0) = v_0(x). 
\end{cases}
\end{align}
Then there exists a constant $C $ which depends upon $p,r$ and $\sigma$, such that 
$$
\| v (t) \| _{B^\sigma _{p,r}} \le e^{V(t) } \left( \| v (0) \| _{B^\sigma _{p,r}} + \int_0^t e^{-V(s) } \| g(s) \| _{B^\sigma _{p,r}} ds \right) ,
$$
where $V(t) = \int_0^t \|  \partial _x f \| _{B^{\sigma -1}_{p,r}} ds$.  If $  - \min\{ \frac1p, 1-\frac1p\} \le \sigma < 1+\frac1p$,  the above estimate holds with 
$V(t) = \int_0^t \|  \partial _x f \| _{B^{ 1/p}_{p,\infty} \cap L^\infty} ds$.
\end{proposition}  
%
%
\section{Proof of Theorem \ref{thm1}} 
Our proof of non-uniform dependence relies upon choosing two sequences of initial data,  $u_{0,n} (x)$ and $v_{0,n}(x)$, which converge to each other as $n$ tends towards $\infty$ in $B^s_{p,r}$, however, the corresponding solutions, denoted $u_n$ and $v_n$, remain bounded away from each other for any positive time. We begin by outlining the proof. 

We let $\widehat \phi (\xi)$ be a smooth bump function equal to $1$ when $|\xi| \le 1/4$ and support on the interval $|\xi|< \frac12$, and we choose  two constants $\delta$ and $\sigma$ which satisfy  $ 0 < \delta<\frac18$, $ \max\{ 1+\frac1p, s- \frac98\}  < \sigma < s-1$ and $0< \frac{8\delta}{p} < s-\sigma -1$. For integers $n \ge 10$, we choose the initial data 
\begin{align}
u_{0,n} (x) = 2^{-ns -\frac{\delta}{p} n } \phi(2^{-\delta n} x) \cos \left ( \frac{17}{12}  2^n x \right) , \quad v_{0,n} (x) = u_{0,n} (x) + 2^{-  \frac12 n  } \phi \left (2^{- \delta n } x   \right )  ,
\end{align}
which satisfies the following estimate. 
\begin{lemma} \label{LYZ} 
For any $s,{s'}  \in \mathbb R$ and  $(p,r ) \in [1,\infty] \times [1, \infty)$, we have 
\begin{align}&
\|   2^{-  \frac12 n   }  \phi(2^{- \delta n } x   )  \|_{B^{{s'} } _{p,r} } =  2^{ \frac{\delta}{p} n-  \frac12 n } \| \phi \| _{B^{{s'} } _{p,r} } 
\\
&
  \| 2^{-ns - \frac {\delta }{p} n}  \phi(2^{-\delta n} x) \cos (  \frac{17}{12} 2^{  n}  x)   \|_{B^{{s'} } _{p,r} }\le 2^{n ( {s'}  - s) } \| \phi\|_{B^{{s'} } _{p,r} } 
\\
&
  \| 2^{-ns - \frac {\delta }{p} n}  \phi(2^{-\delta n} x) \sin (  \frac{17}{12} 2^{  n}  x)   \|_{B^{{s'} } _{p,r} }\le 2^{n ( {s'}  - s) } \| \phi\|_{B^{{s'} } _{p,r} } 
\\
&
 \liminf _{n\rightarrow \infty}    \| 2^{-ns - \frac {\delta }{p} n}  \phi^3(2^{-\delta n} x) \sin (  \frac{17}{12} 2^{  n}  x)   \|_{B^{s } _{p,r} }  \approx 1  .
\end{align}     
\end{lemma} 
\begin{proof}
Using
$$\mathcal F \left( \phi(2^{-\delta n} \cdot)\right) (\xi ) =  2^{\delta n}\widehat \phi (2^{\delta n} \xi)$$ 
we see that $supp ( \mathcal F [  \phi( 2^{-\delta n}   x) ])   \subseteq  \{ \xi:   |\xi| \in [    - \frac1{2^{\delta n+1}}  ,  \frac1{2^{\delta n+1} } ] \}$, from which we can see 
$$
\Delta _q  \phi(2^{-\delta n} x) = \begin{cases}
\phi(2^{-\delta n} x)  \ \text{ if } \  q = -1, 
\\
0 \ \text{ if } q \ge 0 .
\end{cases} 
$$
Therefore, 
$$
\|   2^{-  \frac12 n - \frac{\delta}{p} n}  \phi(2^{- \delta n } x   )  \|_{B^{{s'} } _{p,r} }  =2^{-{s'}  } 2^{-  \frac12 n - \frac{\delta}{p} n}  \|     \phi(2^{- \delta n } x   )  \|_{ L^p  }  =  2^{-{s'}  } 2^{-  \frac12 n  } \|   \phi   \| _{L^p } ,
$$ 
which completes the first estimate in the lemma. Next, we compute the support of \\
$\mathcal F \left[   \phi(2^{-\delta n} x) \cos (  \frac{17}{12} 2^{  n}  x)  \right] (\xi)$,
 and the support of $  ( \mathcal F [  \phi( 2^{-\delta n}   x) \sin( \frac{17}{12} 2^n x)]) $ and show that they both are subsets of  $  \{ \xi:   |\xi| \in [  \frac{17}{12} 2^{n  }   - \frac1{2^{1+\delta n} }  ,  \frac{17}{12} 2^{n   }  +\frac1{2^{1+\delta n} } ] \}$. 
Indeed, using the identity $\sin(\theta) = \frac12 i( e^{-i\theta} - e^{i\theta} )$ we have 
\begin{align} 
 \mathcal F [  \phi(2^{-\delta n} x) \sin( \frac{17}{12} 2^n x)]  & = \int e^{-ix\xi } \phi(2^{-\delta n} x) \sin(2^n x) dx 
 \\
 & 
 = \frac{i  } 2  \int (e^{-i   x(\xi  +   \frac{17}{12} 2^n)}- e^{-i x(\xi  -   \frac{17}{12} 2^n)}  ) \phi( 2^{-\delta n }  x)  dx
 \\
 & 
 = \frac{i 2^{\delta n} } 2  \int (e^{-i 2^{\delta n} x(\xi  +   \frac{17}{12} 2^n)}- e^{-i2^{\delta n} x(\xi  -   \frac{17}{12} 2^n)}  ) \phi( x)  dx  ,
\end{align} 
which clearly equals $\frac{i2^{\delta n} }2\widehat \phi( 2^{\delta n } \xi+ \frac{17}{12} 2^{ ( 1+\delta ) n} ) - \frac{i2^{\delta n}}2\widehat \phi(2^{\delta n } \xi- \frac{17}{12} 2^{ ( 1+\delta ) n} ) $. A similar calculation shows 
$ \mathcal F [  \phi(x) \cos( \frac{17}{12} 2^n x)]  = \frac{2^{\delta n}}2\widehat \phi(2^{\delta n }\xi+ \frac{17}{12} 2^{ ( 1+\delta ) n} ) ) + \frac{2^{\delta n}}2\widehat \phi(2^{\delta n} \xi -  \frac{17}{12} 2^{ ( 1+\delta ) n} ) ) $, and therefore it has the same support. From here it is easy to see that  
$$
\Delta _q  \phi(2^{-\delta n} x)  \cos (2^n x) = \begin{cases}
\phi(2^{-\delta n} x)  \cos ( \frac{17}{12} 2^n x)  \ \text{ if } \  q = n, 
\\
0 \ \text{ if } q \neq n.
\end{cases} 
$$
Therefore, we may compute 
\begin{align}
 2^{-ns - \frac {\delta }{p} n}    \|  \phi(2^{-\delta n} x) \cos (  \frac{17}{12} 2^{  n}  x)   \|_{B^{{s'} } _{p,r} }
 & = 
 2^{-ns - \frac {\delta }{p} n}   2^{n{s'}  }    \|  \phi(2^{-\delta n} x) \cos ( \frac{17}{12}  2^{  n}  x)   \|_{ L^p }
 \\
 & \le 
 2^{n({s'}  -s)  - \frac {\delta }{p} n}     \|  \phi(2^{-\delta n} x)  \|_{ L^p } =  2^{n({s'}  -s) } \|  \phi( x)  \|_{ L^p } ,
 \end{align}
 which completes the second and third estimates. 
 We now consider the fourth quantity and notice that a similar calculation shows that 
 $$
  \mathcal F [  \phi^3 (2^{-\delta n} x) \sin( \frac{17}{12} 2^n x)]    = \frac{i2^{\delta n} }2\widehat { \phi^3} ( 2^{\delta n } \xi+ \frac{17}{12} 2^{ ( 1+\delta ) n} ) - \frac{i2^{\delta n}}2\widehat { \phi^3} (2^{\delta n } \xi- \frac{17}{12} 2^{ ( 1+\delta ) n} ) .
 $$
 Therefore,  we can compute the $B^s_{p,r}$ norm by
 \begin{align}
 2^{-ns - \frac {\delta }{p} n}    \|  \phi ^3(2^{-\delta n} x) \sin (  \frac{17}{12} 2^{  n}  x)   \|_{B^{s} _{p,r} }
 & = 
 2^{- \frac {\delta }{p} n}       \|  \phi ^3(2^{-\delta n} x) \sin (  \frac{17}{12} 2^{  n}  x)   \|_{ L^p } . 
 \end{align}
A similar calculation from the second estimate  shows that  the last line is bounded from above.    A change of variables yields 
 $$ 
  \|   \phi^3(2^{-\delta n} x) \sin (  \frac{17}{12} 2^{  n}  x)   \|_{ L^p  }  ^p =  2^{\delta n } \int _{-\infty}^\infty   |  \phi^{3} (  x) \sin ( \frac{17}{12}  2^{  n+\delta n}  x)   | ^p  dx , 
 $$ 
 and a simple calculation shows the remaining integral is bounded from below by a constant independent of $n$. 
\end{proof} 
We summarize several important estimates in the following corollary.
\begin{corollary} \label{cor}
For any ${s'}  \in \mathbb R$, $s>2+1/p$ and  $(p,r ) \in [1,\infty] \times [1, \infty)$, we have 
 \begin{align*}&
\| u_{0,n} \| _{L^\infty} \approx 2^{-n(s + \frac{\delta }{p} )} , \quad 
  \| u_{0,n} \| _{ B^{s'} _{p,r} }  \lesssim 2^{ n({s'}  - s )  }  , 
 \quad 
 \| v_{0,n} \| _{ B^{s'} _{p,r} }  \lesssim 2^{n({s'}  -   s ) }  + 2^{n (\frac{\delta}{p} -\frac12 )} ,  
\\
&  
 \| v_{0,n} \| _{L^\infty} \approx  2^{ - \frac12 n }  , \quad   \| \partial _x  v_{0,n} \| _{ B^{s'} _{p,r} }  \lesssim    2^{n(    \frac{ \delta} {p}  -\frac12 - \delta )  } + 2^{n({s'}  +1 - s)}  ,
   \quad  
 \| \partial _x  v_{0,n} \| _{  L^\infty }  \lesssim        2^{-(\frac12  + \delta) n }   ,
 \\ 
 &
  \| \partial _x   ^2 v_{0,n} \| _{ B^{s'} _{p,r} }  \lesssim   2^{-(\frac12  -  \frac{ \delta} {p}  + 2 \delta ) n } +  2^{n( {s'}  +2 -s )} 
  ,
 \quad \| \partial _x   ^2 v_{0,n} \| _{ L^\infty  }  \lesssim   2^{-( \frac12 +2 \delta ) n } +  2^{n( 2 -  s-\frac{\delta}{p} )}   .
  \notag
\end{align*} 
\end{corollary} 

We let $T < \inf \{ T_n\} $, $T_n$ the minimum of the lifespan of the solutions $u_n$ and $v_n$ (the well-posedness argument ensures that this infimum  is positive, so long as the initial data is uniformly bounded). 
Our proof proceeds by showing the following results.
\begin{enumerate}
\item  $\| u_{0,n} \|_{B^s_{p,r}}  \approx 1$, and for all for all $t \in [0,T]$,  $  \lim_{n\rightarrow \infty} \| u_n - u_{0,n} \|_{B^s_{p,r}} = 0$ 
\item  $\| v_{0,n} \|_{B^s_{p,r}}  \approx 1$, $ \lim_{n\rightarrow \infty} \| u_{0,n} - v_{0,n}  \|_{B^s_{p,r}}   =0$, and for all for all $t \in [0,T]$,  
$ \lim_{n\rightarrow \infty} \| v_n -  w_n \|_{B^s_{p,r}} =   c t^{\beta} $, $1<  \beta $,
where the  {\em approximate solution}, $w_n(x,t)$, is defined as
\begin{align}
w_ n = v_{0,n} +  t  v_{0,n} ^2  \partial _x v_{0,n}  .
\end{align}
\item
We will show that there exists a positive constant $c$, such that for any small positive time
\begin{align}
\liminf_{n\rightarrow \infty} 
 \| w_n - u_{0,n} \|_{B^s_{p,r}}  \ge c t. 
\end{align} 
\item 
By the triangle inequality, this yields that  for any small positive $t$
\begin{align}
 \| u_n - v_{n} \|_{B^s_{p,r}}   \ge \| w_n - u_{0,n} \|_{B^s_{p,r}}  -   \| u_n - u_{0,n} \|_{B^s_{p,r}} - \| v_n -  w_n \|_{B^s_{p,r}} . 
 \end{align} 
Combining the previous results, we can conclude that there exists a constant $c$ such that  for any small $ t>0$, $ \liminf_{n\rightarrow \infty}  \| u_n - v_{n} \|_{B^s_{p,r}}  \ge c t$, which proves the theorem. 
\end{enumerate}

\vskip0.1in
\noindent
{\bf Proof of (1).} By Lemma   \ref{LYZ}, $\| u_{0,n} \|_{B^s_{p,r}}  = C$, where $C$ is a constant which depends only upon $\phi$. The local well-posedness result tells us that for all $n$, there exists $T>0$ with $u_n \in C([0,T]; B^s_{p,r})$. We re-write the FORQ equation as 
\begin{align}
 u_t  + ( u^2  - (\partial _x u) ^2 ) \partial _x u + N(u) = 0.
\end{align}
We estimate $ \| u_n - u_{0,n} \|_{B^s_{p,r}} $ by setting $\widetilde u_n = u_n - u_{0,n}$ and notice that $\widetilde u_n $ satisfies the following transport equation
\begin{align} \notag
  \partial_t \widetilde u _ n + ( u_n^2  - (\partial_x u_n)^2 )  \partial_x \widetilde u_n  = - ( u_n^2  - (\partial_x u_n)^2 )  \p_x u_{0,n}  -N(u_n)     .
\end{align}
We now apply Proposition \ref{transport_prop} to conclude  that 
\begin{align}\label{3.7} 
\|  \widetilde u \|_{B^{s-1} {p,r} } \le e^{CV(t) } \int_0^t e^{-C V(\tau) } \| ( u_n^2  - (\partial_xu_n)^2 ) \p_x u_{0,n}    + N(u_n)\| _{B^{s-1} _{p,r} } d\tau ,
\end{align}
where $V(t) = \int_0^t \| \partial _x  ( u_n^2  - (\partial_xu_n)^2 ) \|_{B^{{s-1} -1} _{p,r} }  d\tau$.

Using property (4) of Lemma \ref{properties}  we have 
\begin{align} 
 \| \partial _x  ( u_n^2  -  (\partial_xu_n)^2  ) \|_{B^{{s-2} } _{p,r} }  \le 
  \|   ( u_n^2  -  (\partial_xu_n)^2   ) \|_{B^{{s-1} } _{p,r} }  \le  \| u_n \|_{B^{{s-1} } _{p,r} } ^2  +   \| \partial_ x u_n \|_{B^{{s-1} } _{p,r} } ^2 \le 2\| u_n \|_{B^{s} _{p,r} } ^2,
\end{align} 
so $V(t)$ is bounded for all $t \in [0, T]$. 
Next, we use property (5) of Lemma \ref{properties} to estimate 
\begin{align} \notag
  \|  ( u_n^2  - (\partial_x u_n)^2 ) \partial_x u_{0,n}  \|_{B^{{s-1} } _{p,r} }
\le    \|  u_n^2  - (\partial_xu_n)^2 \|_{B^{{s-1} } _{p,r} }  \|  \partial _x u_{0,n}   \|_{ L^\infty } + 
  \|  u_n^2  - (\partial_x u_n)^2 \|_{ L^\infty } 
 \|  \partial _x u_{0,n}   \|_{B^{{s-1} } _{p,r} } .
\end{align}
We  use  property (4) in Lemma \ref{properties}  and the estimate found in Lemma \ref{LYZ} and its corollary to find $\| ( \partial_xu_n)^2 \|_{ L^\infty } \le \|  (\partial_xu_n)^2  \|_{ B^{{s-2}   }_{p,r}  } \lesssim  \|  u_{0,n} \|_{ B^{{s-1}   }_{p,r}  }  ^2 \approx 2^{ -2n } $. Using these estimates, we find 
\begin{align} \label{first_term} 
  \|  ( u_n^2  - (\partial_xu_n)^2   ) \p_x u_{0,n}   \|_{B^{{s-1} } _{p,r} }
\lesssim     2^{n(1-s +\frac{\delta}{p} ) } + 
 2^{ -2} 
\lesssim 2^{n \beta },
\end{align}
where  $\beta = \max\{  1+\frac{\delta}{p} -s, -2\}  <-1$. 
Next we estimate 
\begin{align}
  \|  N(u_n)  \|_{B^{{s-1} } _{p,r} }  \le  
    \|    \frac23 u_n ^3  +u_n(\partial_x u_n) ^2  \|_{B^{{s-1} -1} _{p,r} }  +   \|     \frac13 ( \partial_ xu_n)^3 \|_{B^{{s-1} -2} _{p,r} } . 
\end{align}
We use property (6) of Lemma \ref{properties}  to estimate the second term on the right hand side, while we use the algebra property (4) to estimate the first term on the right hand side of the above inequality 
\begin{align}
  \|  N(u_n)  \|_{B^{{s-1} } _{p,r} }  \lesssim
    \|   u_n    \|_{B^{{s-2}} _{p,r} } ^3 +\| u_n  \|_{B^{{s-2} } _{p,r} } \|  \partial_x u_n   \|_{B^{{s} -2} _{p,r} } ^2 +   \|    \partial_ xu_n  \|_{B^{\bar  \sigma -1} _{p,r} } ^2   \|    \partial_ xu_n  \|_{B^{\bar\sigma-2} _{p,r} }, 
\end{align}
where $\bar \sigma = \max\{{s-1} , 3/2\}$.
Using Corollary \ref{cor}  we find
\begin{align} \label{second_term}
  \|  N(u_n)  \|_{B^{{s-1} } _{p,r} }  \lesssim
    2^{-6n  } +     2^{-2n  }  2^{ -n }  +       2^{n(2 \bar\sigma  -2s) } 2^{n( \bar\sigma  -1-s) } \lesssim   2^{-2n  } . 
\end{align}
Combining estimates \eqref{first_term} and \eqref{second_term}  with inequality \eqref{3.7} we find 
\begin{align} 
\|  \widetilde u  _ n\|_{B^{s-1} _{p,r} } \lesssim         2^{n \beta  } + 2^{ -2  n}   .
\end{align}
From the well-posedness argument, we have the following estimate on the solution
\begin{align} 
\|    u _ n\|_{B^{s+1} _{p,r} } \lesssim    \|    u _ {0,n} \|_{B^{s+1}_ {p,r} }  \approx 2^{n} , 
\end{align}
which, by the triangle inequality,  implies $\|  \widetilde u  _ n\|_{B^{s+1} _ {p,r} }   \lesssim2^{n}$.  Now, applying the interpolation inequality, property (7) of Lemma \ref{properties}, we find 
\begin{align} 
\|   \widetilde u _ n\|_{B^{s } _{p,r} } \lesssim    \|   \widetilde u _ n\|_{B^{s-1 } _{p,r} }  ^{1/2} \|   \widetilde u _ n\|_{B^{s+1 } {p,r} }  ^{1/2} \lesssim   2^{ \frac12 \beta n}2^{\frac12n} , 
\end{align}
which clearly tends towards zero as $n$ tends towards infinity, and this completes the proof of (1).

\vskip0.1in
\noindent
{\bf Proof of (2).}   It is clear that   and $ \lim_{n\rightarrow \infty} \| u_{0,n} - v_{0,n}  \|_{B^s_{p,r}}   =0$. We must  evaluate 
$ \lim_{n\rightarrow \infty} \| v_n -  w_n \|_{B^s_{p,r}}  $
where $w_ n = v_{0,n} -  t  v_{0,n} ^2 \partial _x v_{0,n}  .$  We denote $\widetilde w_n = v_n -  w_n$ and we find that
\begin{align}
\partial_ t \widetilde w_n   + ( v_n^2  - (  \partial_x v_n)^2 )  \partial_x \widetilde w_n  =  - ( v_n^2  - (  \partial_x v_n)^2 ) \partial _x w_ n   -N(v_n)  + v_{0,n} ^2 \partial _x v_{0,n}  ,
\end{align}
and after adding and subtracting $v_{0,n}^2 \partial_x w_n$ this is equivalent to
\begin{align}\notag
\partial_ t \widetilde w_n   +  ( v_n^2  - (  \partial_x v_n)^2 )     \partial_x \widetilde w_n & =  (  \partial_x v_n)^2 \partial_x w_n  -  \widetilde w_n  ( v_n + v_{0,n} )\partial _x w_ n 
\\&
+ t v_{0,n}^2  \partial_x v_{0,n}  ( v_n + v_{0,n} ) \partial _x w_ n   + t v_{0,n}^2 \partial_x ( v_{0,n}^2\partial_x v_{0,n}) -N(v_n)     .
\end{align}
We now apply Proposition \ref{transport_prop} to conclude  that 
\begin{align} \notag
\|  \widetilde w \|_{B^\sigma_{p,r} } &\le e^{CV(t) } \int_0^t e^{-C V(\tau) } \|   ( \partial _ x v_n)^2 \partial_x w_n    -  N(v_n)   -  \widetilde w_n  ( v_n + v_{0,n} )\partial _x w_ n 
 \| _{B^\sigma_{p,r} }  
 \\ & +
 \tau \|    v_{0,n}^2  \partial_x v_{0,n}  ( v_n + v_{0,n} ) \partial _x w_ n   +  v_{0,n}^2 \partial_x ( v_{0,n}^2\partial_x v_{0,n})
 \| _{B^\sigma_{p,r} } 
 d\tau ,
 \label{3.18} 
\end{align}
where $V(t) = \int_0^t \| \partial _x   ( v_n^2  - (  \partial_x v_n)^2 )  \|_{B^{\sigma-1} _{p,r} }  d\tau$. 
As in the proof of (1), $V(t)$ is bounded and we use Corollary \ref{cor} to find 
\begin{align} 
\| N(v_n)\| _{B^\sigma_{p,r} }  & \lesssim
    \|   v_n    \|_{B^{\sigma-1} _{p,r} } ^3 +\| v_n  \|_{B^{\sigma-1} _{p,r} } \|  \partial_x v_n   \|_{B^{\sigma-1} _{p,r} } ^2 +   \|    \partial_ xv_n  \|_{B^{\bar \sigma-1} _{p,r} } ^2   \|    \partial_ xv_n  \|_{B^{\bar\sigma-2} _{p,r} }
    \\ 
    & \lesssim 2^{3n ( \frac{\delta}{p} - \frac12) } + 2^{n(  \frac{\delta}{p} - \frac12) } \left( 2^{n ( \frac{\delta}{p} - \frac12 - \delta  ) } + 2^{n(\sigma - s )} \right)^2 + \left( 2^{n ( \frac{\delta}{p} - \frac12 - \delta  ) } + 2^{n( \bar \sigma - s )} \right)^3
\end{align}
where $\bar \sigma = \max\{ \sigma , \frac32\} $. Using $\delta <  \frac18$, we see that $\delta/p - 1/2  <- 3/8$ and therefore  we obtain 
\begin{align} 
\| N(v_n)\| _{B^\sigma_{p,r} }  &  \lesssim 2^{ -\frac98 n   } + 2^{ -\frac38 n  } \left( 2^{ -\frac12 n   } + 2^{n(\sigma - s )} \right)^2 +  2^{ -\frac32 n  }   \lesssim
2^{ -\frac98 n   } .
 \end{align}
  
  Next we estimate $  \|   ( \partial _ x v_n)^2\partial_x w_n   \| _{B^\sigma_{p,r} }   =   \|   ( \partial _ x v_n)^2\partial_x  \left( v_{0,n} -  t  v_{0,n} ^2 \partial _x v_{0,n} \right)   \| _{B^\sigma_{p,r} } $. We break this into two pieces and find 
\begin{align} \notag
  \| &   ( \partial _ x v_n)^2\partial_x  \left( v_{0,n}   \right)   \| _{B^\sigma_{p,r} }
  \lesssim
   \| \partial_x v_n\| _{B^\sigma_{p,r} }  \| \partial_x v_n \partial_x v_{0,n} \| _{L^\infty} 
  +  
   \| \partial_x v_n\| _{ L^\infty }  \| \partial_x v_n \partial_x v_{0,n} \| _{B^\sigma_{p,r} }
   \\
   & \lesssim 
   \|   v_n\| _{B^{\sigma+1}_{p,r} }  \|   v_n\| _{B^{\sigma}_{p,r} }\|  \partial_x v_{0,n} \| _{L^\infty} 
  +  
   \|  \partial_ xv_n\| _{B^{\sigma-1}_{p,r} } \left( \|  \partial _x v_n\|  _{B^{\sigma-1}_{p,r} } \|\partial_x v_{0,n} \| _{B^{\sigma}_{p,r} } + \|   v_n\|  _{B^{\sigma+1}_{p,r} } \|\partial_x v_{0,n} \| _{ L^\infty } \right).
   \notag
\end{align}
Now, using the estimates found in Corollary \ref{cor} we have 
\begin{align}
  \|   ( \partial _ x v_n)^2\partial_x  \left( v_{0,n}   \right)   \| _{B^\sigma_{p,r} }
  & \lesssim  2^{n(\sigma+1-s)} 2^{n(\frac{\delta}{p} - \frac12)} 2^{-n(\frac12+\delta)} 
  \notag
  \\ & 
  +      \|  \partial _x v_n\|  _{B^{\sigma-1}_{p,r} } 
  \left(     \|  \partial _x v_n\|  _{B^{\sigma-1}_{p,r} } 2^{n(\sigma +1 -s)} + 2^{n(\sigma +1 -s)}  2^{-n( \frac12 + \delta) } 
  \right).
  \label{estimate_w_p} 
\end{align}
In order to continue, we must use the estimate for $  \|  \partial _x v_n\|  _{B^{\sigma-1}_{p,r} } $ found in the following lemma. 
\begin{lemma}
For  $ 0 < \delta<\frac18$, $ \max\{ 1+\frac1p, s- \frac98\}  < \sigma < s-1$ and $0< \frac{8\delta}{p} < s-\sigma -1$, 
we have 
 $  \|  \partial _x v_n\|  _{B^{\sigma-1}_{p,r} } \lesssim  \|  \partial _x v_{0,n}\|  _{B^{\sigma-1}_{p,r} }  \approx 2^{n(\frac\delta{p} -\frac12 -\delta )} $
\end{lemma}
\begin{proof}
Notice $ \omega = \partial_x v_n$ is a solution to the following transport equation
$$
\partial_t \omega +(v_n^2 - (\partial_x v_n)^2) \partial_x \omega = -\partial_ x (v_n^2 - (\partial_x v_n)^2) \omega - \partial_x N(v_n) . 
$$
We apply Proposition \ref{transport_prop}  to find 
$$
\| \omega\| _{B^{\sigma -1 } _{p,r} } \le e^{ V(t) } \left( \| \omega (0) \| _{B^{\sigma -1 } _{p,r} } + \int _0^t e^{-V(s)} \|
\partial_ x (v_n^2 - (\partial_x v_n)^2) \omega +  \partial_x N(v_n)\| _{B^{ \sigma -1} _{p,r} }   ds\right),
$$ 
where $V(t) = \int_0^t \|  \partial _x  (v_n^2 - (\partial_x v_n)^2)  \| _{B^{ 1/p}_{p,\infty} \cap L^\infty} ds$ . Clearly $V(t)$ is bounded independent of $n$. We use the algebra property to find 
$$
\|
\partial_ x (v_n^2 - (\partial_x v_n)^2) \omega  \| _{B^{ \sigma -1} _{p,r} } 
\lesssim 
\|
  (v_n^2 - (\partial_x v_n)^2) \| _{B^{ \sigma } _{p,r} } \|  \omega  \| _{B^{ \sigma -1} _{p,r} }  \lesssim  \|  \omega  \| _{B^{ \sigma -1} _{p,r} } .
$$
We estimate $\| \partial_x N(v_n)\| _{B^{ \sigma -1} _{p,r} }  $ by 
\begin{align}
\| \partial_x N(v_n)\| _{B^{ \sigma -1} _{p,r} } 
& \lesssim
 \|  v_n^3 \| _{B^{ \sigma -1} _{p,r} } +  \|  v_n(  \partial_ x v_n )^2\| _{B^{ \sigma -1} _{p,r} }  +  \|  ( \partial_ xv_n) ^3 \| _{B^{ \sigma -2} _{p,r} }  
\\  &
 \lesssim
 2^{-\frac98 n}  + \| \omega \| _{B^{ \sigma -1} _{p,r} }  .
\end{align}
Making these substitutions into the above, we now have 
$$
\| \omega\| _{B^{\sigma -1 } _{p,r} } \le  \| \omega (0) \| _{B^{\sigma -1 } _{p,r} } + \int _0^t    2^{-\frac98 n } +   \| \omega\| _{B^{\sigma -1 } _{p,r} }ds ,
$$ 
and now applying Gr\"onwall's inequality yields the result. 
\end{proof}
Applying this lemma, our inequality \ref{estimate_w_p} becomes 
\begin{align}
  \|   ( \partial _ x v_n)^2\partial_x  \left( v_{0,n}   \right)   \| _{B^\sigma_{p,r} }
  & \lesssim  2^{n(\sigma+1-s)} 2^{n(\frac{\delta}{p} - \frac12)} 2^{-n(\frac12+\delta)} 
  \notag
  \\ & 
  +      2^{n( \frac\delta{p}-\frac12 - \delta)}
  \left(    2^{n( \frac\delta{p}-\frac12 - \delta)}
 2^{n(\sigma +1 -s)} + 2^{n(\sigma +1 -s)}  2^{-n( \frac12 + \delta) } 
  \right) 
  \\
  & \approx  2^{n( \sigma+\frac{ \delta}{p} - \delta-s)} + 
  2^{n( \sigma+\frac{ 2\delta}{p} -2 \delta-s)}  \approx  2^{n( \sigma+\frac{ \delta}{p} - \delta-s)}.
  \label{terrible} 
\end{align}
Next we estimate 
 $  \|   ( \partial _ x v_n)^2\partial_x  \left(   t  v_{0,n} ^2 \partial _x v_{0,n} \right)   \| _{B^\sigma_{p,r} }$ and find 
 \begin{align} 
   \|   ( \partial _ x v_n)^2\partial_x  \left(   t  v_{0,n} ^2 \partial _x v_{0,n} \right)   \| _{B^\sigma_{p,r} } 
   & \lesssim 
     \|   ( \partial _ x v_n)^2 \| _{B^{\sigma-1} _{p,r} } \| \partial_x  \left(   t  v_{0,n} ^2 \partial _x v_{0,n} \right)   \| _{B^\sigma_{p,r} }
  \notag  \\
    & + 
      \|   ( \partial _ x v_n)^2 \| _{B^\sigma_{p,r} }\| \partial_x  \left(   t  v_{0,n} ^2 \partial _x v_{0,n} \right)   \| _{   L^\infty} .
      \notag
\end{align}
We use the previous lemma   to get 
 \begin{align} 
   \|   ( \partial _ x v_n)^2\partial_x  \left(   t  v_{0,n} ^2 \partial _x v_{0,n} \right)   \| _{B^\sigma_{p,r} } 
   & \lesssim 
2^{n(\frac{2\delta}{p} - 1  -2\delta )}  \| \partial_x  \left(   t  v_{0,n} ^2 \partial _x v_{0,n} \right)   \| _{B^\sigma_{p,r} }
   \notag \\&     + 
      \|   ( \partial _ x v_n)^2 \| _{B^\sigma_{p,r} } \| \partial_x  \left(   t  v_{0,n} ^2 \partial _x v_{0,n} \right)   \| _{   L^\infty}.
\end{align}
  We use 
    $\|  \partial_x  \left(   t  v_{0,n} ^2 \partial _x v_{0,n}   \right)  \| _{   L^\infty} 
 \lesssim  t \|     v_{0,n} (  \partial _x v_{0,n})^2     \| _{   L^\infty}  +t \|     v_{0,n} ^2 (  \partial _x ^2v_{0,n})     \| _{   L^\infty}$
 and we use  \\
   $\|  \partial_x  \left(   t  v_{0,n} ^2 \partial _x v_{0,n}   \right)  \| _{B^\sigma_{p,r} }
 \lesssim  t \|     v_{0,n} (  \partial _x v_{0,n})^2     \| _{B^\sigma_{p,r} }
+t \|     v_{0,n} ^2      \| _{   L^\infty} \|     (  \partial _x ^2v_{0,n})     \| _{  B^\sigma_{p,r}}+
 \|     v_{0,n} ^2      \| _{  B^\sigma_{p,r} } \|     (  \partial _x ^2v_{0,n})     \| _{  L^\infty} $
 and apply the estimates found in Corollary \ref{cor} to conclude 
  \begin{align}  \notag
   \|   ( \partial _ x v_n)^2\partial_x  \left(   t  v_{0,n} ^2 \partial _x v_{0,n} \right)   \| _{B^\sigma_{p,r} } 
   & \lesssim 
 t 2^{n(\frac{2\delta}{p} - 1  -2\delta )}\left( 2^{n( 2\sigma +\frac32 +\frac{\delta}{p} -2s ) } + 
 2^{n(\sigma +1 -s)} 
 \right. \\ & +\left.
 2^{n(\frac{2\delta}{p} -1 ) } ( 2^{-( \frac12 +2 \delta ) n } +  2^{n( 2 -  s-\frac{\delta}{p} )}) 
 \right) \notag
  \\&    + 
    t 2^{n(2\sigma +2 -2s)}  \left( 2^{n( -\frac32 -2\delta) } + 2^{n(1-s-\frac{\delta}{p} )}
    \right) 
      \notag
      \\& \lesssim   t 2^{n( \sigma+\frac{ \delta}{p} - \delta-s)}.
      \label{terrible2} 
\end{align}
Combining estimates \ref{terrible}  and \ref{terrible2}, we obtain 
$$
\| (\partial_x v_n)^2 \partial_x w_n \| _{B^\sigma _{p,r} } \lesssim   2^{n( \sigma+\frac{ \delta}{p} - \delta-s)} .
$$ 
  
  In order to estimate the next term in inequality \ref{3.18} we shall need to estimate $\|\partial_x w_n \|_{B^\sigma _{p,r} } $. We have 
\begin{align} 
      \|    \partial _x  w_n   \|_{B^{\sigma} _{p,r} } &  \lesssim    
   \|   \partial_ x   v_{0,n}    \|_{B^{\sigma} _{p,r} }  +t   \|     v_{0,n}    \|_{B^{\sigma} _{p,r} }    \|   \partial_x    v_{0,n}    \|_{B^{\sigma} _{p,r} } ^2+ t   \|     v_{0,n}    \|_{B^{\sigma} _{p,r} } ^2   \|   \partial_x    v_{0,n}    \|_{B^{\sigma+1} _{p,r} }  
  \\
  & 
  \lesssim
    2^{-( \frac12  -  \frac{ \delta} {p}  + \delta ) n } + 2^{n( \sigma   +1 - s)}   + t 2^{n(\frac{\delta}{p} - \frac12) } \left(      2^{-( \frac12  -  \frac{ \delta} {p}  + \delta ) n } + 2^{n( \sigma   +1 - s)}    \right) 
    \\ & + t 2^{ 2 ( \frac{\delta}{p} - \frac12 ) n} \left (     2^{-( \frac12  -  \frac{ \delta} {p}  + \delta ) n } + 2^{n( \sigma   +2 - s)}     \right)  ,
\end{align}
and using the    condition $0< \frac{2\delta}{p} < s-\sigma -1$ we see that all of the exponents are negative, and therefore $     \|    \partial _x  w_n   \|_{B^{\sigma} _{p,r} }  \lesssim 1$. 
Using this estimate, we now find 
\begin{align}
   \|  \widetilde w_n  ( v_n + v_{0,n} )\partial _x w_ n  
 \| _{B^\sigma_{p,r} } 
 &   \lesssim 
   \|  \widetilde w_n  \| _{B^\sigma_{p,r} }  \|   v_n + v_{0,n}  \| _{B^\sigma_{p,r} } \| \partial _x w_ n   \| _{B^\sigma_{p,r} }  
   \\
   &   \lesssim \left( 2^{n({s'}  -   s ) }  + 2^{ (\frac{\delta}{p} -\frac12 )n} \right)  \| \partial _x w_ n   \| _{B^\sigma_{p,r} }     \|  \widetilde w_n  \| _{B^\sigma_{p,r} }\lesssim   \|  \widetilde w_n  \| _{B^\sigma_{p,r} } .
  \end{align} 
  We now estimate $ \|    v_{0,n}^2  \partial_x v_{0,n}  ( v_n + v_{0,n} ) \partial _x w_ n    
 \| _{B^\sigma_{p,r} } $. Using the algebra property,  and Corollary \ref{cor}  we find 
 \begin{align}
  \|    v_{0,n}^2  \partial_x v_{0,n}  ( v_n + v_{0,n} ) \partial _x w_ n    
 \| _{B^\sigma_{p,r} } 
 & \le 
   \|    v_{0,n} \| _{B^\sigma_{p,r} } ^2\|   \partial_x v_{0,n}   \| _{B^\sigma_{p,r} } \| ( v_n + v_{0,n} ) \| _{B^\sigma_{p,r} } \| \partial _x w_ n    
 \| _{B^\sigma_{p,r} }
 \\
 & \lesssim 
 \left( 2^{n(\sigma -   s ) }  + 2^{ (\frac{\delta}{p} -\frac12 )n} \right) ^3 \left ( 2^{n(\sigma +1  -   s ) }  + 2^{ (\frac{\delta}{p} -\frac12 -\delta  )n} \right) 
 \\
 & 
 \lesssim  2^{n(3 \sigma - 3s )} + 2^{-\frac98 n} .
  \end{align} 
We estimate $\|   v_{0,n}^2 \partial_x ( v_{0,n}^2\partial_x v_{0,n})
 \| _{B^\sigma_{p,r} } 
$
using  properties (4) and (5) of Lemma \ref{properties}  and find
 \begin{align}
 \|   v_{0,n}^2 \partial_x ( v_{0,n}^2\partial_x v_{0,n})
 \| _{B^\sigma_{p,r} }   & \lesssim   
  \|   v_{0,n}  \| _{B^\sigma_{p,r} }   ^3\|  \partial_x v_{0,n}
 \| _{B^\sigma_{p,r} }   ^2 +   \|   v_{0,n} ^4 \| _{L^\infty} \|  \partial_x^2  v_{0,n}
 \| _{B^\sigma_{p,r} }   
 +   \|   v_{0,n} ^4 \| _{B^\sigma_{p,r} }   \|  \partial_x^2  v_{0,n}
 \| _{L^\infty}    ,  \notag
  \end{align} 
  and using the estimates found in Corollary \ref{cor} we have
  \begin{align}\notag
 \|   v_{0,n}^2 \partial_x ( v_{0,n}^2\partial_x v_{0,n})
 \| _{B^\sigma_{p,r} }   
 & \lesssim   
 \left(  2^{n(\sigma-   s ) }  + 2^{n (\frac{\delta}{p} -\frac12 )} \right)    ^3 
 \left(  2^{n(    \frac{ \delta} {p}  -\frac12 - \delta )  } + 2^{n({\sigma}  +1 - s)}
 \right)   ^2
 \\&
  +  2^{-2n}  \left( 2^{-(\frac12  -  \frac{ \delta} {p}  + 2 \delta ) n } +  2^{n( {\sigma}  +2 -s )} 
 \right)  + 2^{4 n(\frac{\delta}{p} - \frac12 ) } 
 \left(2^{-( \frac12 +2 \delta ) n } +  2^{n( 2 -  s-\frac{\delta}{p} )}   \right)  \notag 
 \\
 & \lesssim 2^{n(\sigma - s) }  + 2^{ -\frac98 n}  \lesssim  2^{n(\sigma - s) }   .
  \end{align} 
Combining the previous estimates, and substituting into inequality \eqref{3.18}  we find
\begin{align}  
\|  \widetilde w \|_{B^\sigma_{p,r} } &\lesssim \int_0^t  \left(   2^{n( \sigma+\frac{ \delta}{p} - \delta-s)}+  \|  \widetilde w_n  (\tau)  \| _{B^\sigma_{p,r} }  +
 \tau 2^{n(\sigma - s)}  \right)  d\tau . 
\end{align} 
By Gr\"onwall's inequality, we have 
\begin{align}  
\|  \widetilde w \|_{B^\sigma_{p,r} } &\lesssim    2^{n( \sigma+\frac{ \delta}{p} - \delta-s)} + 2^{n( \sigma - s)} t ^2  .
 \label{3.22} 
\end{align}
A simple calculation also shows that
$$
\|   \widetilde  w _n \|_{B^{s+2} _ {p,r} } \lesssim 2^{2n}  . 
$$
We now apply the interpolation lemma using the last inequality and inequality \eqref{3.22} to find 
\begin{align}
\|   \widetilde  w _n \|_{B^{s} _ {p,r} } &  \lesssim \|   \widetilde  w _n \|_{B^{\sigma} _  {p,r} } ^\theta \|   \widetilde  w _n \|_{B^{s+2} _ {p,r} } ^{1-\theta} 
\\
& \lesssim   \left(     2^{n( \sigma+\frac{ \delta}{p} - \delta-s)} + 2^{n( \sigma - s)} t ^2    \right)^\theta 2^{2n (1-\theta) }  , 
\end{align} 
where $\theta =2 (2+s-\sigma )^{-1}  > 1/2 $.  Therefore, we find
\begin{align}
\|   \widetilde  w _n \|_{B^{s} _ {p,r} }  
& \lesssim
   \left[ \left(     2^{n(2 \sigma+\frac{ 2\delta}{p} -2 \delta-2s)} + 2^{2\sigma + \frac{\delta}{p} - \delta -2s } t^2 + 2^{n( 2\sigma - 2s)} t ^4   \right)   2^{n(2s-2\sigma)    } \right] ^{ (2+s-\sigma )^{-1} }  
         \\
   & =  
       \left[ \left(     2^{n( \frac{ 2\delta}{p} -2 \delta)} + 2^{n(   \frac{\delta}{p} - \delta ) } t^2 +  t ^4   \right)       \right] ^{ (2+s-\sigma )^{-1} }  .
\end{align} 
Now taking the liminf of both sides, we see 
\begin{align}
 \liminf_{n\rightarrow \infty} \|   \widetilde  w _n \|_{B^{s} _ {p,r} }  
& \lesssim
 t^{ 4 (2+s-\sigma )^{-1} }   = t^{2 \theta} .
\end{align}

{\bf Proof of (3).} 
We now show that  for small $t\ge 0$,
\begin{align}
\liminf_{n\rightarrow \infty} 
 \| w_n - u_{0,n} \|_{B^s_{p,r}}  \ge c t. 
\end{align}  
We have 
$$
 \| w_n - u_{0,n} \|_{B^s_{p,r}}  =  \|  2^{-\frac12 n} \phi  (2^{-\delta n } \cdot) -  t  v_{0,n} ^2 \partial _x  v_{0,n} \|_{B^s_{p,r}}  , 
$$
and by the triangle inequality
\begin{align}\notag
 \| w_n - u_{0,n} \|_{B^s_{p,r}} & \ge   t   \|   v_{0,n} ^2 \partial _x  v_{0,n} \|_{B^s_{p,r}}     -\|  2^{-\frac12 n} \phi   (2^{-\delta n } \cdot)  \|_{B^s_{p,r}} .
\end{align} 

Now we break  $\|   v_{0,n} ^2   \partial _x  v_{0,n} \|_{B^\sigma_{p,r}} $ into pieces. Let $v_{0,n} = u_{0,n} + u_\ell$, where $u_\ell$ is the low frequency part of the initial data. We have 
\begin{align} 
\|   v_{0,n} ^2   \partial _x  v_{0,n} \|_{B^s_{p,r}}  \ge 
\|     u_\ell ^2   \partial _x u_{0,n} \|_{B^s_{p,r}} -\|     u_\ell ^2   \partial _x  u_\ell\|_{B^s_{p,r}} - 2 \|     u_{0,n}   u_\ell     \partial _x  v_{0,n} \|_{B^s_{p,r}} -
 \|     u_{0,n}  ^2   \partial _x  v_{0,n} \|_{B^s_ {p,r}} .  
 \notag
\end{align} 
The last three terms above tend to zero as $n \rightarrow \infty$. Indeed, using property (4) of Lemma \ref{properties} we have 
\begin{align}
\|     u_\ell ^2   \partial _x  u_\ell\|_{B^s_{p,r}}  \lesssim \|     u_\ell\|_{B^s_{p,r}}  ^2\|    \partial _x  u_\ell\|_{B^s_{p,r}}  \lesssim 2^{ \frac{2\delta}{p} n- n } 2^{n(\frac\delta p - \frac12 - \delta )} ,
\end{align}
which tends to zero as $n \rightarrow \infty$.  
Similarly, using property (5) of Lemma \ref{properties} we can bound the third term on the right hand side by
\begin{align}  
  \|     u_{0,n}   u_\ell     \partial _x  v_{0,n} \|_{B^s_{p,r}} & \lesssim 
    \|     u_{0,n}   u_\ell   \| _{L^\infty} \|    \partial _x  v_{0,n} \|_{B^s_{p,r}} 
    +
    \|     u_{0,n}   u_\ell   \| _{B^s_{p,r}}  \|    \partial _x  v_{0,n} \|_{ L^\infty } 
    \\
    & \lesssim 
    2^{-(s+\frac12 )n}  2^n + 2^{ \frac{\delta}{p} n-\frac12 n } \left( 2^{-( \frac12 +\delta ) n} + 2^{-(s-1) n } \right),
\end{align} 
which again tends to zero as $n \rightarrow \infty$.  
Finally, 
\begin{align}  
  \|     u_{0,n}  ^2   \partial _x  v_{0,n} \|_{B^s_ {p,r}}  &  \lesssim  \|     u_{0,n}  \|_{ L^\infty}  ^2 \|    \partial _x  v_{0,n} \|_{B^s_ {p,r}} 
  +  \|     u_{0,n}  \|_{B^s_ {p,r}}   ^2 \|    \partial _x  v_{0,n} \|_{ L^\infty} 
\\&   \lesssim  2^{-2s n -\frac{2\delta}{p} n}  2^{n} +  2^{-(\frac12 +\delta)n} ,
\end{align} 
and again, this tends towards zero as $n$ grows towards infinity. 
The first term,  $\|     u_\ell ^2   \partial _x  u_{0,n} \|_{B^s_{p,r}}$, is bounded below by a constant by the last estimate in Lemma \ref{LYZ}. 
  Combining the above estimates, the proof of (3)  is complete. 

{\bf Proof of (4).} 
By the triangle inequality, this yields that  for any positive $t$
\begin{align}
 \| u_n - v_{n} \|_{B^s_{p,r}}   \ge \| w_n - u_{0,n} \|_{B^s_{p,r}}  -   \| u_n - u_{0,n} \|_{B^s_{p,r}} - \| v_n -  w_n \|_{B^s_{p,r}} . 
 \end{align} 
Combining the previous results, we can conclude that there exists a constant $c$ such that  for any $t \in [0, T]$, $ \liminf_{n\rightarrow \infty}  \| u_n - v_{n} \|_{B^s_{p,r}}  \ge c_1 t - c_2 t^{2\theta} $, which proves the theorem since $2\theta >1$.

 \bibliographystyle{plain}
 
\bibliography{mybib}

\end{document}